\documentclass{amsart}

\usepackage{amsmath}
\usepackage{amsfonts}
\usepackage{amssymb}
\usepackage{amsthm}
\usepackage{enumerate}
\usepackage[usenames,dvipsnames]{xcolor}
\usepackage{graphicx}

\newtheorem{defin}{Definition}[section]
\newtheorem{thm}[defin]{Theorem}
\newtheorem{cor}[defin]{Corollary}

\newtheorem{lem}[defin]{Lemma}
\newtheorem{prop}[defin]{Proposition}

\newcommand{\Z}{\mathbb{Z}}

\begin{document}
\title[\resizebox{4.7in}{!}{A family of polycyclic groups over which the uniform conjugacy problem is NP-complete}]{A family of polycyclic groups over which the conjugacy problem is NP-complete}
\author{Bren Cavallo}
\address{Bren Cavallo, Department of Mathematics, CUNY Graduate Center, City University of New York}
\email{bcavallo@gc.cuny.edu}

\author{Delaram Kahrobaei}
\address{Delaram Kahrobaei, CUNY Graduate Center, PhD Program in Computer Science and NYCCT, Mathematics Department, City
University of New York} \email{dkahrobaei@gc.cuny.edu}

\begin{abstract}
In this paper we study the conjugacy problem in polycyclic groups.  Our main result is that we construct polycyclic groups $G_n$ whose conjugacy problem is at least as hard as the subset sum problem with $n$ indeterminates.  As such, the uniform conjugacy problem over the groups $G_n$ is NP-complete where the parameters of the problem are taken in terms of $n$ and the length of the elements given on input.
\end{abstract}

\maketitle

\section{Introduction}
\nocite{*}
Given a finitely presented group $G$, the conjugacy decision problem for $G$ asks if two elements $u, v \in G$ are conjugate.  Along with the word problem and isomorphism problem, it was one of the original group theoretic problems introduced by Dehn in 1911.  In the context of non-commutative group based cryptography, one often studies the search variant: given two conjugate elements in $G$, find a third element of $G$ conjugating one to the other.  In 1999 Anshel, Anshel, and Goldfeld \cite{anshel1999algebraic} created a key exchange protocol that relied on solving the conjugacy search problem multiple times and proposed braid groups as the platform group.   In the years since, different parties \cite{garber2006length, myasnikov2007length} have shown that heuristic attacks can in fact break the protocol when it is done over braid groups.  Other cryptographic protocols using the conjugacy search problem include \cite{kahrobaei2009decision, kahrobaei2006non, DBLP:journals/corr/abs-1210-7493, ko2000new}.

\

In 2004 Eick and Kahrobaei \cite{eick2004polycyclic} proposed polycyclic groups as a secure platform for AAG and offered computational evidence.  Later Garber, Kahrobaei, and Lam \cite{garber2013analyzing} experimentally showed that polycyclic groups were resistant to many of the heuristic attacks that are strong against braid groups.  In this paper we offer theoretical evidence that the conjugacy decision and search problems over polycyclic groups are difficult.  We construct polycyclic groups $G_n$ whose conjugacy search and decision problems are at least as hard as the subset sum search and decision problems in $n$ indeterminates which is well known to be NP-complete.   The $G_n$ also have the additional property that algebraic computations such as conjugation and collection can be performed quickly when group elements are represented by exponent vectors.  In this way, a polynomial time algorithm that solves either conjugacy problem in these groups would imply P = NP.

\

We devote {\bf Sections 2, 3}, and {\bf 4} to the conjugacy problem in groups, preliminaries on polycyclic groups, and the subset sum problem.  In {\bf Section 5} we introduce the twisted subset sum problem ($TSSP$) and discuss a dynamic programming algorithm that solves it. In {\bf Section 6} we explicitly construct the groups $G_n$ that we will be working over and show that the $TSSP$ reduces to the conjugacy decision problem in the $G_n$.  In {\bf Section 7} we show that the conjugacy problem in the $G_n$ is in NP.  We also note that the same computations can also be used to show that multiplication and collection can also be performed in polynomial time.  In {\bf Section 8} we reduce the subset sum problem to the $TSSP$ which implies that the conjugacy problem in the $G_n$ is NP-complete.

\section{The Complexity of Conjugacy Problem in Polycyclic Groups}
The conjugacy problem in a finitely presented group, $G$, can be defined as follows: given two elements $u,v \in G$ determine if there exists $x \in G$ such that $xux^{-1} = v$. For a general group, this problem is known to be undecidable. See \cite{collins1977conjugacy} for an example of a group with an undecidable conjugacy problem that contains an index 2 subgroup with decidable conjugacy problem.

\

In the case of polycyclic groups the conjugacy problem was shown to be decidable independently by Formanek \cite{formanek1976conjugate} and Remesslennikov \cite{remeslennikov1969conjugacy}. More precisely, they showed that any virtually polycyclic group is conjugacy separable- if $u,v \in G$ are not conjugate, then there exists a finite homomorphic image in which the images of $u$ and $v$ are not conjugate. This allows us to create an algorithm for the conjugacy problem by performing two simultaneous tasks. We conjugate $u$ by elements of $G$ and check if the conjugation is equal to $v$. This can be done in polycyclic groups by comparing normal forms of elements. We also simultaneously enumerate all homomorphisms from $G$ into a finite group and check if the images of $u$ and $v$ are conjugate. At least one of these processes will eventually terminate, and we would have our answer. Clearly, this algorithm may take very long in general and the nature of the algorithm does not lend itself to complexity analysis.

\

Later on, Eick and Ostheimer (see \cite{eick2001algorithms, eick2003orbit}) developed a more practical algorithm for solving the conjugacy problem in polycyclic groups, but did not perform a complexity analysis. It was pointed out in \cite{eick2004polycyclic} that the algorithm of Eick and Ostheimer is likely slow since it may involve computation of the unit group of an algebraic number field. Also in \cite{eick2004polycyclic}, Eick and Kahrobaei tested the algorithm on metabelian polycyclic groups and noted that as the Hirsch length increased, the average time it took to solve the conjugacy increased dramatically. In a group of Hirsch length 14, the average time the algorithm took was over 100 hours.

\

More recently, Sale \cite{sale2011short, sale2012conjugacy} studied the geometry of conjugacy in solvable groups and found upper bounds on the conjugacy length function for certain polycyclic groups. The conjugacy length function, $CLF_G (n)$ of a group $G$ is defined as follows:
$$CLF_G (n) = \max \{ \min \{ |w| : wu = vw \} : |u| + |v| \leq n, \, u \, \mbox{and} \, v \, \mbox{are conjugate in} \, G\}$$
In his work, Sale specifically studied groups of the form $\Z^n \rtimes_\phi \Z^k$ with certain technical conditions on $\phi$. He showed that under these conditions, if $k = 1$ then $CLF_G(n)$ is at most linear in $n$ and that if $k > 1$, $CLF_G (n)$ is at most exponential in $n$. Note that in this case $|u|$ and $|v|$ are the geodesic lengths of $u$ and $v$ respectively. Also note that this result depends on the presentation of $G$, which Sale fixes.

\

In this vein, Sale showed that you can solve the conjugacy problem in such groups by conjugating $u$ by all words of length less that or equal to $CLF_G (|u| + |v|)$ to attempt to find an element that conjugates $u$ to $v$. If no such word is found, then $u$ and $v$ are not conjugate. In this case, we can also check equality of words in $G$ by comparing normal forms of elements. These bounds give a better idea of how long a brute force algorithm would take in these groups would take. Namely, we know that to perform such an algorithm, at most $g(CLF_G (|u| + |v|))$ potential conjugators must be checked, where $g$ is the growth function of $G$. As such, in a polycyclic group with exponential growth, Sale's results say that for $k = 1$ at most an exponential words need to be checked while for $k > 1$ the upper bound is double exponential.

\

Aside from Sale's work, very little is actually known about the complexity of the conjugacy problem in non-abelian polycyclic groups. Based on experiments in \cite{eick2004polycyclic, garber2013analyzing} it is conjectured that existing algorithms are super polynomial and that the complexity increases with respect to Hirsch length and length of words.   It is still open whether or not there is a polynomial time algorithm for the conjugacy with respect to Hirsch length and length of elements.  Also, very little known about the complexity of the conjugacy problem if we fix the Hirsch length and study a single polycyclic group and no polynomial time algorithm has been found in that case.  Additionally, complexity of algorithms in the cases of variations in which the polycyclic group is given by an arbitrary presentation is unknown.

\

We close this section by mentioning some recent results about the complexity of the conjugacy problem in other solvable groups. In \cite{sale2013geometry}, Sale found upper bounds for the conjugacy length functions of free solvable groups, namely that they each have a cubic upper bound. In the same paper, he also found an upper bound for the conjugacy length function of the wreath product of two groups in terms of their own conjugacy length functions.  Similarly, Vassileva found \cite{vassileva2011polynomial} a polynomial time algorithm for the conjugacy problem in free solvable groups.  She also found that if two groups, $A$ and $B$, have polynomial time algorithms for their conjugacy problems, and if $B$ has a polynomial time algorithm for its power problem, then there is a polynomial time algorithm for the conjugacy problem in $A \wr B$.  In \cite{diekert2013conjugacy}, Diekert, Miasnikov, and Wei{\ss} studied the circuit complexity of the conjugacy problem in the Baumslag - Solitar group, $BS(1,2)$ and showed that its conjugacy problem is $TC^0$ - complete.  They also showed that the conjugacy problem in the Baumslag - Gersten group, $G(1, 2)$, has a polynomial time algorithm in a strongly generic setting.  Namely, they show that for almost all inputs, the conjugacy problem can be decided in polynomial time.

\section{Poly-$\Z$ Groups}
A group $G$ is \emph{polycyclic} if it has a subnormal series with cyclic quotients.  Namely, $G$ has subgroups $ G_0, G_1,  \cdots, G_{n}$ such that

\begin{equation} \{1\} = G_{0} \triangleleft G_1 \triangleleft \cdots \triangleleft G_n  = G \end{equation}

\

and $G_{i+1}/G_{i}$ is a cyclic group.

\

In this section, we will summarize a variety of results on polycyclic groups that can be found in \cite{drutu2011lectures, eick2001algorithms} that will be used in the remainder in this paper.

\

Given a subnormal series as in (1), one can find a \emph{polycyclic generating set}, $\{g_1, \cdots, g_n\}$ where $G_{i} = \langle g_i, G_{i-1} \rangle$ for $1 \leq i \leq n$.  With respect to this generating set, each group element $g \in G$ can be represented as $g_1^{k_1} \cdots g_n^{k_n}$ and such a representation is called its \emph{normal form}.   We also call each $g_i^k$ a \emph{syllable}. For every polycyclic group, there exists a polycyclic generating set such that any word has a unique normal form.  The process of converting an arbitrary word to its normal form is called \emph{collection}.

\

Of specific interest to us in this paper are polycyclic groups where $G_i/G_{i-1} \simeq \Z$ for each $i$.  Such a group is called poly-$\Z$ as each quotient is isomorphic to $\Z$.  If this is the case, then $G$ is obtained from the final non-trivial group in the subnormal series $G_1 \simeq \Z$ by successive semi-direct products with $\Z$ as follows.  It is a standard result (see \cite{drutu2011lectures} for instance) that if $G/G_{n-1} \simeq \Z$ then $G \simeq G_{n-1} \rtimes_\phi \Z$.  If we take $g_n$ as the generator of $\Z$, then $\phi$ is given by conjugating elements of $G_{n -1}$ by $g_n$.  For reference, $G_{n-1} \rtimes_\phi \Z$ is the group with elements $w g_n^k$ with $w \in G_{n-1}$ and multiplication given by:

$$(w g_n^k)(w' g_n^l) = w \phi^k (w') g_n^{k + l}$$

Proceeding inductively, we see that $G$ can be written as:

\begin{equation} ( \cdots ((\Z \rtimes_{\phi_1} \Z) \rtimes_{\phi_2} \Z) \rtimes_{\phi_3} \cdots ) \rtimes_{\phi_{n-1}} \Z \end{equation}

where $\phi_j$ is conjugation by $g_{j+1}$.

\

The groups we will be interested in will be constructed in this fashion by explicitly describing the different $\phi_i$.  See \cite{DEPS13} in which the authors use the same construction for more details.  In the following, we will take $g_i$ as the generator of the $i^{th}$ $\Z$ in the semi-direct product form.  From the multiplication rules of the semi-direct product, one can see that $g_j g_i = \phi_{j-1}(g_i) g_j$ for $i < j$.  By using this identity, it is possible to then put any arrangement of letters into normal form so that any $g_i$ appears to the left of $g_j$ when $i < j$.  We also define the \emph{Hirsch length} as the number of $\Z$'s in the semi-direct product formulation of the poly-$\Z$ group.  The Hirsch length is an isomorphism invariant, so while different automorphisms in the construction of the poly-$\Z$ group may lead to  isomorphic groups, the number of factors is necessarily the same.

\

Any word $w = g_1^{k_1} \cdots g_n^{k_n}$ can be represented uniquely by its \emph{exponent vector}, $[k_1, \cdots, k_n]$.  We can then take the length of $w$ to be the length of its corresponding exponent vector, $l(w) = \sum_{i = 1}^n log(|k_i|) = O(n log(K))$ where $K$ is an upper bound of the absolute value of all of the exponents.  This measure of length, is somewhat different than many of the standard ones used when studying algorithms in groups.  Most often, one would take the length of the word to be its distance from the identity in the Cayley graph equipped with the word metric.  In this scenario, we are considering normal forms of group elements, rather than geodesic forms, because they are often used for cryptographic and other algorithmic applications in polycyclic groups.  For instance, in a practical setting, one cannot necessarily generate random words that are of the shortest length and the complexity of converting a word in a polycyclic group into a shortest length representative is unknown.  Additionally, it is more practical when possible to work with group elements as a tuple of exponents rather than generator by generator, making the size of the exponent vector a natural measure length.  In the groups we will be working with, algebraic operations (multiplication, collection, conjugation) can be computed effectively with a Turing machine (see {\bf Section 7}) when inputs are taken in terms of exponent vectors.  This is not necessarily the case in other scenarios where group elements are dealt with generator by generator.  It is also important to note that this measure of length is necessary for the results in the following sections to hold.

\section{Subset Sum Problem}
The subset sum problem, or $SSP$, is the following:  given a set of integers, $L = \{k_1, k_2, \cdots , k_n\}$, and an integer, $M$, determine if there exists subset of $L$ that sums to $M$.   This can also be rephrased as determining if there is a solution to the equation:
$$k_1 x_1 +\cdots + k_n x_n = M \,\,\, \mbox{where} \,\,\, x_i \in \{0,1\}.$$

We can bound the size of the problem from above by considering the length of the list, $n$, and an upper bound on the absolute value of the entries, $K$.  In doing so, the length of the problem can be seen to have length $O(nlog(K))$.  We will also label this instance of the $SSP$, $SSP(L, M)$, or just $(L, M)$ when there is no ambiguity.

\

The $SSP$ is NP-complete, meaning that the existence of a deterministic polynomial-time Turing machine that solves it would imply that P = NP.  In fact, it was originally introduced by Karp as one of his 21 NP-complete problems.  Despite it being NP-complete, there exists a pseudo-polynomial algorithm via dynamic programming (see \cite{kellerer2004knapsack}).  Namely there exists a deterministic algorithm that runs in polynomial time when the length of the problem is taken in terms of the actual numerical entries of the list rather than the number of digits needed to represent them.  As such, the existence of such an algorithm and the NP-completeness of the problem doesn't imply that P = NP.

\

For the purposes of cryptography, we consider the search version of the $SSP$:  given that a subset of $L$ sums to $M$, actually find such a subset.  From the outset, it is not immediately clear how the two problems are related, but one can show that a polynomial time algorithm for one would lead to a polynomial time algorithm for the other.  First we show how an algorithm for the decision problem can be applied at most $n - 1$ times to make an algorithm for the search problem.  This can be done by first checking if the $SSP$, $(L \setminus \{k_n\}, M)$ has a solution.  If not, then we know $k_n$ is a part of our solution and proceed by checking  $(L  \setminus \{k_{n-1}, k_n\}, M - k_n)$.  Otherwise, we know we can create a solution without $k_n$ and proceed by checking $(L  \setminus \{k_{n-1}, k_n\}, M)$.  By doing this repeatedly, we will have eventually found a subset summing to $M$.  In the worst case scenario, we will have reached the end of the list and performed the decision algorithm $n - 1$ times.  Note that since we are in an instance of the search problem, a solution is assumed to exist, so it is not necessary to run the decision algorithm on $(L, M)$.

\

On the other hand, if there were a polynomial time algorithm that solved the search version of the $SSP$, we could use it to prove existence of a polynomial time algorithm for the decision problem.  Rather than give a formal proof, we will just sketch one omitting certain details.  Given a polynomial time Turing machine, $M$, for the $SSP$ search problem, there exists a polynomial $P(n)$ such that for any input $x$, the number of steps $M$ takes on input $x$ is less than or equal to $P(|x|)$.  We can then create another Turing machine, $M'$, that on input $y$ performs the same steps as $M$ for $P(|y|)$ steps.  If $M'$ has not yet finished, it then hits its final state and outputs ``no" as the answer.  $M'$ is then a polynomial time Turing machine for any instance of the $SSP$ decision problem.  Either it terminates in less than $P(|y|)$ steps in which case $M'$ has found a solution for $y$ and outputs ``yes" or $M'$ takes longer than $M$ would if there were a solution, implying that there isn't one, and so $M'$ outputs ``no".  As such, a polynomial time algorithm for either the $SSP$ or its search variant would imply existence of a polynomial time algorithm for the other.

\

We also will use the notion of a polynomial time reduction.  We say that a decision problem, $Q$, can be reduced to a decision problem, $R$, in polynomial time if there exists a polynomial time mapping, $f$, from instances of $Q$ to instances of $R$, such that an instance $x$ is a ``yes" instance of $Q$  if and only if $f(x)$ is a ``yes" instance of $R$.  Such a mapping also must only increase the lengths of instances at most polynomially.   We also write $Q \leq_p R$ to say that $Q$ polynomial time reduces to $R$.

\

 If such an $f$ exists, then a polynomial time algorithm for $R$ would imply a polynomial time algorithm for $Q$.  Given an instance $x$ of $Q$, we can compute $f(x)$ and then perform our polynomial time decision algorithm for $R$.  As such, a decision problem, $A \in \mbox{NP}$, is NP-complete if $B \leq_p A$ for all $B \in \mbox{NP}$ and one can prove a problem $C \in \mbox{NP}$, is NP-complete if $A \leq_p C$ where $A$ is NP-complete.  Finally, note that polynomial time reductions are transitive: $A \leq_p B$ and $B \leq_p C$ imply $A \leq_p C$.

\section{The Twisted Subset Sum Problem}

Given a list $L = \{k_1, \cdots , k_n \}$ and an integer $M$, the twisted subset sum problem ($TSSP$) is determining if the following equation has a solution:
$$k_n x_n (-1)^{x_1 + x_2 + \cdots x_{n - 1}} + k_{n-1} x_{n-1} (-1)^{x_1 + x_2 + \cdots + x_{n-2}} + \cdots + k_2 x_2 (-1)^{x_1} + k_1 x_1 = M$$

where $x_i \in \{0,1\}$.  Note that we could trivially let $x_i$ be any number and replace $x_i$ with $x_i \bmod{2}$ in the above equation.  For the remainder of the paper, we write $x'_i = x_i \, \bmod{2}$.

\

Before continuing, we would like to remark that dynamic programming leads to a polynomial time solution to the $TSSP$ when the coefficients are taken to be in unary and briefly describe the algorithm.    In short, we create an array, $A$, where rows go over $1$ through $n$ and the columns go over all possible sums that can be made by the list.  At each entry in the array, $A(i , j)$, we ask if there is a solution to $TSSP(\{k_1, \cdots , k_i\}, j)$ and also record whether or not we achieved our sum from adding or subtracting the final number.

\

Note that unlike with the standard dynamic programming algorithm for the subset sum problem, we will potentially put more than one item into each space in the array.  This is because, the way you arrive at a sum depends on what value the following number will be.  Therefore, it may be important to distinguish between arriving at a certain number after a subtraction versus arriving at the same number after an addition.  As such, it might be technically best to think of the array as a multi-dimensional array, but for practicality we will consider a two dimenisonal array of lists.

\

\begin{enumerate}[1.]
\item Create an array, $A$, where the rows are labeled $1$ through $n$ and the columns are labeled $-S$ through $S$ where $S = |k_1| + \cdots + |k_n|$.
\item Set $A(1, j) := (T, 1)$ if $j = k_1$, otherwise, set $A(1, j)$ to $F$.  The $1$ indicates that the next item will be subtracted.  In the future we may also see $(T, 0)$ to indicate that the next item will be added.  This corresponds to the sum of the $x_i$ modulo $2$.
\item For each $A(i, j)$ where $i > 1$ do the following (we do this recursively so we compute row 2 before row 3):
\begin{enumerate}[i.]
\item If $k_i = j$ add a $(T, 1)$ into $A(i, j)$.
\item If $(T, \epsilon) \in A(i - 1, j)$ where $\epsilon \in \{-1, 1\}$, then add a $(T, \epsilon)$ into $A(i, j)$.  As such, $A(i -1, j) \subset A(i,j)$.
\item If $(T, 0) \in A(i - 1, j - k_i)$ add $(T, 1)$ to $A(i, j)$.
\item If $(T, 1) \in A(i - 1, j + k_i)$ add $(T, 0)$ to $A(i, j)$.
\item Otherwise set $A(i, j)$ equal to $F$.
\end{enumerate}
\end{enumerate}

\

In order to solve $TSSP(\{k_1, \cdots , k_n\}, M)$ we fill the array, $A$,  and check if $A(n, M)$ contains either $(T, 0)$ or $(T, 1)$.  Note that filling this array can be done in polynomial time when the coefficients are taken in unary as we do polynomially many elementary operations $n \times 2S$ times. On the other hand, if the coefficients are taken in binary, $S$ is exponential in the size of the coefficients and so filling the entire array would be of exponential time complexity.  We later show in {\bf Section 7} that the $TSSP$ is NP-Complete when the entries are taken in binary, meaning that there exists a polynomial algorithm solving it if and only if P = NP.

\section{The conjugacy problem over the groups $G_n$}

In this section, we show how any instance of the $TSSP$ where the set of integers has length $n$, can be turned into an instance of the conjugacy problem of a polycyclic group with Hirsch length $2n + 1$.  Since the reduction is polynomial, we prove that $TSSP$ polynomial time reduces to the conjugacy problem over the $G_n$.

\

The group $G_n$ will be constructed as follows:

$$( \cdots ((\Z \rtimes_{\phi_1} \Z) \rtimes_{\phi_2} \Z) \rtimes_{\phi_3} \cdots ) \rtimes_{\phi_{2n}} \Z$$

where $$\phi_{2i - 1} (g_j) =
  \begin{cases} g_1^{-1} & \text{if } j = 1 \\
g_j &\text{otherwise}
\end{cases}$$

\

 and

$$\phi_{2i} (g_j) =
  \begin{cases} g_1 g_j & \text{if } j = 2i \\
g_j &\text{otherwise }
\end{cases}$$

\

The multiplicative structure of $G_n$ can then be seen as such: if $j$ is even, $g_j g_1 = g_1^{-1} g_j$ and $g_{j+1} g_j = g_1 g_j g_{j+1}$. The following lemma, is a consequence of these multiplicative identities and will assist us in collecting words throughout the paper.

\

\begin{lem}
Let $j$ be even and $a,b \in \Z$.  Then $g_j^a g_1^b = g_1^{b(-1)^a} g_j^a$ and $g_{j+1}^a g_j^b = g_1^{a b'} g_j^{b} g_{j+1}^a$
\end{lem}

\begin{proof}
To prove the first identity, just check that $g_j^a g_1^b = \phi_{j  - 1}^a (g_1^b) g_j^a = (\phi_{j  - 1}^a (g_1))^b g_j^a = (g_1^{(-1)^a})^b  g_j^a =  g_1^{b(-1)^a} g_j^a$.  For the second identity, we first check the case $a = 1$.  We then have $g_{j+1} g_j^b = (\phi_j (g_j))^b  g_{j + 1} = (g_1 g_j)^b g_{j + 1}$.  Now note that $(g_1 g_j)^2 = g_1 g_j g_1 g_j = g_1 g_1^{-1} g_j g_j = g_j^2$ from the first part of the lemma.  Therefore if $b = 2l$ is even, then $(g_1 g_j)^b = ((g_1 g_j)^2)^l = (g_j^2)^l = g_j^b$.  Using the same logic, if $b$ is odd,  $(g_1 g_j)^b = g_1 g_j^b$.  Therefore, $g_{j+1} g_j^b = (g_1 g_j)^b g_{j + 1} = g_1^{b'} g_j^b g_{j + 1}$.

\

To continue with the more general case, $g_{j+1}^a g_j^b = g_{j+1}^{a - 1} g_1^{b'} g_j^b g_{j+1} = g_1^{b'} g_{j+1}^{a - 1} g_j^b g_{j+1}$ and after iterating this computation one obtains $g_{j+1}^a g_j^b = g_1^{a b'} g_j^{b} g_{j+1}^a$.

\end{proof}

We now show that the $TSSP$ instance, $TSSP(\{k_1, k_2 \cdots , k_n\}, M)$, has a solution if and only if $g_3^{k_1} g_5^{k_2} \cdots g_{2n + 1}^{k_n} \sim g_1^{-M} g_3^{k_1} g_5^{k_2} \cdots g_{2n + 1}^{k_n}$.  Our general strategy will be as follows: conjugate  $g_3^{k_1} g_5^{k_2} \cdots g_{2n + 1}^{k_n}$ by a generic word in $G_n$ and collect.  First, note that from the multiplication rules above, it can be seen that any $g_i$ where $i$ is odd, commutes with $g_3^{k_1} g_5^{k_2} \cdots g_{2n + 1}^{k_n}$.  It can also be seen from the lemma that conjugating by one of the generators with even index does not introduce any generators with even index in the collected word.  Therefore, without loss of generality, we can assume that our generic word is of the form $g_2^{x_1} g_4^{x_2} \cdots g_{2n}^{x_n}$ because adding in any generators with odd index doesn't affect the conjugated product.  In {\bf Theorem 6.2}, we prove:

$$(g_2^{x_1} g_4^{x_2} \cdots g_{2n}^{x_n})(g_3^{k_1} g_5^{k_2} \cdots g_{2n + 1}^{k_n})(g_2^{x_1} g_4^{x_2} \cdots g_{2n}^{x_n})^{-1}= $$
$$g_1^{-p(k_1, \cdots, k_n, x_1,  x_2, \cdots, x_n)} g_3^{k_1} g_5^{k_2} \cdots g_{2n + 1}^{k_n}$$

where $p(k_1, \cdots, k_n, x_1,  x_2, \cdots, x_{n}) =$

$$k_n x'_n (-1)^{x_{n-1} + x_{n - 2} + \cdots + x_1} +  k_{n-1} x'_{n-1} (-1)^{x_{n-2} + \cdots + x_1} + \cdots +  k_2 x'_2 (-1)^{x_1} +  k_1 x'_1$$

Notice then, that finding $x_i$ such that $p(k_1, \cdots, k_n, x_1,  x_2, \cdots, x_{n}) = M$ is exactly the $TSSP(\{k_1, k_2 \cdots , k_n\}, M)$. As such, the two words are conjugate if and only if $TSSP(\{k_1, k_2 \cdots , k_n\}, M)$ has a solution.  Additionally, the length of the inputs to the problems are only off by a polynomial. It can be seen that the length of both the $TSSP$ and the conjugacy problem are $O(n \log(K))$ where $K$ and $n$ are chosen as before.  It is also clear that the transformation is efficient to compute.  Since the general conjugacy problem in $G_n$ includes all of these instances we then have that the conjugacy problem in $G_n$ is polynomial time reducible to the $TSSP$ with $n$ indeterminates.

\

 It is worth noting, that here is where our measure of length of group elements is crucial.  If we had the length of elements in $G$ be measured as $|k_1| + \cdots + |k_n|$, then the conjugacy problem would be exponentially larger than an instance of $TSSP$ and we would not have a polynomial time reduction.  Also, as we pointed out in {\bf Section 5}, the $TSSP$ has a polynomial time solution using dynamic programming when the elements in the list are taken in unary, which then implies that the conjugacy problem with the more standard measure of length of elements in normal form would also have a polynomial time solution.

\

\begin{thm}
$TSSP$ with $n$ indeterminates is polynomial time reducible to the conjugacy problem in the group $G_n$ of Hirsch length $2n + 1$.
\end{thm}

\begin{proof}
To prove this, we show that solving the conjugacy problem, $g_3^{k_1} g_5^{k_2} \cdots g_{2n + 1}^{k_n} \sim g_1^{-M} g_3^{k_1} g_5^{k_2} \cdots g_{2n + 1}^{k_n}$, in $G_n$ yields a solution to  TSSP$(\{k_1, k_2 \cdots , k_n\}, M)$.

\

We proceed by induction on $l$ where we conjugate by the last $l$ syllables of the generic word. Rather than starting with $l = 1$ it may clarify the computation to start with $l = 2$.  In this case we collect:

$$(g_{2n -2}^{x_{n - 1}} g_{2n}^{x_n}) (g_3^{k_1} g_5^{k_2} \cdots g_{2n + 1}^{k_n})(g_{2n -2}^{x_{n - 1}} g_{2n}^{x_n}) ^{-1}$$

Conjugating first by  $g_{2n}^{x_n}$, we find:
$$ g_{2n}^{x_n}(g_3^{k_1} g_5^{k_2} \cdots g_{2n + 1}^{k_n}) g_{2n}^{-x_n} =$$
$$g_3^{k_1} g_5^{k_2} \cdots  g_{2n}^{x_n} (g_{2n + 1}^{k_n}  g_{2n}^{-x_n}) =$$
$$g_3^{k_1} g_5^{k_2} \cdots g_{2n}^{x_n} (g_1^{k_n x'_n} g_{2n}^{- x_n} g_{2n + 1}^{k_n}) =$$
$$g_1^{k_n x'_n (-1)^{x_n}} g_3^{k_1} g_5^{k_2} \cdots  g_{2n + 1}^{k_n} =$$
$$g_1^{- k_n x'_n} g_3^{k_1} g_5^{k_2} \cdots  g_{2n + 1}^{k_n}$$

Negt we conjugate by $g_{2n -2}^{x_{n - 1}}$:
$$g_{2n -2}^{x_{n - 1}} (g_1^{- k_n x'_n} g_3^{k_1} g_5^{k_2} \cdots g_{2n - 1}^{k_{n - 1}} g_{2n + 1}^{k_n}) g_{2n -2}^{-x_{n - 1}} =$$
$$g_1^{- k_n x'_n (-1)^{x_{n-1}}} g_3^{k_1} g_5^{k_2} \cdots g_{2n -2}^{x_{n - 1}} (g_{2n - 1}^{k_{n - 1}}  g_{2n -2}^{-x_{n - 1}}) g_{2n + 1}^{k_n} =$$
$$g_1^{- k_n x'_n (-1)^{x_{n-1}}} g_3^{k_1} g_5^{k_2} \cdots g_{2n -2}^{x_{n - 1}} (g_1^{k_{n - 1}x'_{n -1}} g_{2n -2}^{-x_{n - 1}} g_{2n - 1}^{k_{n - 1}} ) g_{2n + 1}^{k_n} =$$
$$g_1^{- k_n x'_n (-1)^{x_{n-1}} - k_{n - 1}x'_{n-1}} g_3^{k_1} g_5^{k_2} \cdots g_{2n - 1}^{k_{n - 1}}  g_{2n + 1}^{k_n} =$$
$$g_1^{-p(k_{n-1},k_n,x_{n-1},x_n)} g_3^{k_1} g_5^{k_2} \cdots g_{2n - 1}^{k_{n - 1}}  g_{2n + 1}^{k_n}$$

We now induct and assume the result holds for $l = n -1$ and show it holds for $l = n$ .  In this case we have:
$$(g_2^{x_1} g_4^{x_2} \cdots g_{n}^{x_{n}})(g_3^{k_1} g_5^{k_2} \cdots g_{2n + 1}^{k_n})(g_2^{x_1} g_4^{x_2} \cdots g_{n}^{x_{n}})^{-1} = $$
$$g_2^{x_1} (g_1^{-p(k_2, \cdots , k_n, x_2, \cdots ,  x_n)} g_3^{k_1} \cdots g_{2n + 1}^{k_n}) g_2^{-x_1}$$  Conjugating by $g_2^{x_1}$ then yields:

$$g_2^{x_1} (g_1^{p(k_2, \cdots , k_n, x_2, \cdots ,  x_n)} g_3^{k_1} \cdots g_{2n + 1}^{k_n}) g_2^{-x_1} = $$
$$g_1^{-p(k_2, \cdots , k_n, x_2, \cdots ,  x_n)(-1)^{x_1}} g_2^{x_1} (g_3^{k_1}  g_2^{-x_1}) \cdots g_{2n + 1}^{k_n}  = $$
$$g_1^{-p(k_2, \cdots , k_n, x_2, \cdots ,  x_n)(-1)^{x_1}}g_2^{x_1} (g_1^{k_1 x'_1} g_2^{-x_1} g_3^{k_1}) \cdots g_{2n + 1}^{k_n} = $$
$$g_1^{-p(k_2, \cdots , k_n, x_2, \cdots ,  x_n)(-1)^{x_1} - k_1 x'_1} g_3^{k_1} \cdots g_{2n + 1}^{k_n} = $$

$$g_1^{-p(k_1, k_2, \cdots , k_n, x_1, x_2, \cdots ,  x_n)}  g_3^{k_1} \cdots g_{2n + 1}^{k_n}$$

\

It is now enough to note that any possible solution to the above conjugacy problem would give you a solution to the equivalent $TSSP$ instance by eliminating all the $g_j$ with $j$ odd and reducing all the exponents modulo 2.  Therefore a ``yes" answer to the conjugacy decision problem using any algorithm would imply the existence of a solution to the $TSSP$.
\end{proof}

\section{The Conjugacy Problem In $G_n$ Is In NP}
In this section we show that the conjugacy problem in the groups $G_n$ can in fact be checked efficiently.  To do this we will find closed form expressions for conjugating a word by a power of a single generator.  These closed form expressions will be effectively computable with group elements in their normal form.  Since conjugating by a single syllable can be done in polynomial time, conjugating by $2n + 1$ of them is also polynomial time.  Therefore, checking conjugacy is efficient.  These methods can also be used to create closed form expressions for multiplying and collecting elements in normal form.

\

When conjugating elements in $G_n$ there are three cases to consider: conjugation by powers $g_1$, conjugation by powers of $g_j$ with $j$ even, and conjugation by powers of $g_l$ where $l$ is odd and larger than 1.

\

For the first case we collect

$$g_1^k (g_1^{k_1} \cdots g_{2n+1}^{k_{2n +1}}) g_1^{-k}$$

Since each of the even $g_j$ invert $g_1$, when we bring the $g_1^{-k}$ to the left we switch the sign of the exponent according to the parity of the exponents of the even indexed $g_j$.  Also, the odd $g_l$ commute with $g_1$, and do not affect the collection process. Therefore we end up with:

\begin{equation} g_1^{k + k_1 - k(-1)^{k_2 + k_4 + \cdots + k_{2n}}} g_2^{k_2} \cdots g_{2n+1}^{k_{2n +1}} \end{equation}

\

The second case is then collecting

$$g_j^k (g_1^{k_1} \cdots g_{2n+1}^{k_{2n +1}}) g_j^{-k}$$

where $j$ is even.

\

We first move the $g_j^k$ right.  Hopping over the $g_1^{k_1}$ may change the sign of the exponent, but after that, each $g_i$ commutes with $g_j$ for $i < j$.  Therefore as a first step we end up with:

$$(g_1^{k_1(-1)^{k}} g_2^{k_2} \cdots g_j^{k + k_j} \cdots  g_{2n+1}^{k_{2n +1}}) g_j^{-k}$$

In moving the $g_j^{-k}$ to the left, the only thing that doesn't commute is $g_{j+1}$.  To hop over $g_{j+1}^{k_{j+1}}$ we can use {\bf Lemma 6.1} and get

$$g_1^{k_1(-1)^k} g_2^{k_2} \cdots g_j^{k + k_j} (g_{j+1}^{k_{j+1}}  g_j^{-k}) \cdots  g_{2n+1}^{k_{2n +1}} =$$
$$g_1^{k_1(-1)^k} g_2^{k_2} \cdots g_j^{k + k_j} (g_1^{k_{j+1} k'} g_j^{-k}  g_{j+1}^{k_{j+1}})   \cdots  g_{2n+1}^{k_{2n +1}} $$

Finally, we move the $g_1^{k_{j+1} k'}$ to the left to end up with:

\begin{equation} g_1^{k_1(-1)^k + k_{j+1} k'(-1)^{k_2 + k_4 + \cdots + k_j + k}} g_2^{k_2} \cdots g_j^{k_j} g_{j+1}^{k_{j+1}}   \cdots  g_{2n+1}^{k_{2n +1}} \end{equation}

The third case is dealt with similarly to the second. When $l > 1$ is odd:

$$g_l^k (g_1^{k_1} \cdots g_{2n+1}^{k_{2n +1}}) g_l^{-k} =$$
$$ g_1^{k_1} \cdots (g_l^k g_{l - 1}^{k_{l - 1}}) g_l^{k_l - k}  \cdots g_{2n+1}^{k_{2n +1}} =$$
$$ g_1^{k_1} \cdots (g_1^{k k'_{l - 1}}g_{l - 1}^{k_{l - 1}} g_l^{k})g_l^{k_l - k}  \cdots g_{2n+1}^{k_{2n +1}} =$$
\begin{equation} g_1^{k_1 + k k'_{l - 1}(-1)^{k_2 + k_4 + \cdots + k_{l - 3}}} g_2^{k_2} \cdots g_{2n+1}^{k_{2n +1}} \end{equation}

\

Since conjugation is done by successively conjugating elements of the form of those in $(3)$, $(4)$, and $(5)$ these closed forms can iteratively perform a general conjugation.  Such a computation can be performed in polynomial time in terms of $n\log(K)$ because computing the normal form after conjugation by each syllable can be done in polynomial time using the closed forms, and need only be performed $n$ times.  This means that we can create a polynomial time verifier for the conjugacy problem in the $G_n$.

\

These normal forms also provide us with the following corollaries that describe conjugation in the group.  The proofs for both statements can be seen directly be inspecting the closed forms above.

\begin{cor}
Let $u, v \in G_n$ where $u = g_1^{e_1} \cdots g_{2n + 1}^{e_{2n + 1}}$ and $v = g_1^{f_1} \cdots g_{2n + 1}^{f_{2n + 1}}$.
\begin{enumerate}[i.]
\item $u \sim v$ implies $e_i = f_i$ for $i \geq 2$.
\item Let $e_i = f_i$ for $i \geq 2$.  If there exists an even $l - 1$ such that $e_{l - 1} = f_{l - 1}$ is odd, the $u$ and $v$ are conjugate.  In fact, one such element that conjugates $u$ to $v$ is:
 $$g_{l}^{(f_1 - e_1)(-1)^{e_2 + e_4 + \cdots +  e_{l - 3}}}$$
\end{enumerate}
\end{cor}

Note that part $ii$ of the above corollary does not include the difficult cases of the conjugacy problem that we saw in {\bf Section 6}.

\

Now we check that there exists a certificate that is of polynomial length with respect to any  instance of the conjugacy problem $G_n$, $u \sim v$.   Let $u = g_1^{e_1} g_2^{e_2} \cdots g_{2n + 1}^{e_{2n + 1}}$ and $v = g_1^{f_1} g_2^{e_2} \cdots g_{2n + 1}^{e_{2n + 1}}$ and $w = g_1^{x_1} \cdots g_{2n + 1}^{x_{2n + 1}}$ such that $wuw^{-1} = v$.  In the case that there exists an odd exponent above an even indexed generator, we have a certificate of polynomial length from part $ii$ of {\bf Corollary 7.1}.  Therefore, we can assume that for all $j$ even, $e_j$ is even.

\

In this case, we also know there exists a conjugator where $x_l = 0$ for all $l$ odd and greater than 1 by inspecting the closed forms from $(5)$.  Additionally we can take $x_j$ to be $0$ or $1$ for $j$ even by looking at the closed forms from $(4)$.  It remains to put bounds on $x_1$.

\

Let $y$ be the exponent above $g_1$ conjugating by $g_2^{x_2} \cdots g_{2n + 1}^{x_{2n + 1}}$. Then by repeated applications of $(4)$, $|y| \leq |e_1| + |e_3| + \cdots + |e_{2n + 1}|$.  From $(3)$, conjugating by $g_1^{x_1}$ either increases the exponent by $2x_1$ or leaves it unchanged.  Therefore if $f_1 = y$, we can take $x_1$ to be 0 and then clearly a certificate has length $O(n)$.  Otherwise, $f_1 = y + 2x_1$ implying that $|x_1| \leq |f_1| + |y| \leq |f_1| +  |e_1| + |e_3| + \cdots + |e_{2n + 1}|$.  This means that the length of a certificate is bounded from above by $\log(|f_1| +  |e_1| + |e_3| + \cdots + |e_{2n + 1}|) + n$ which is of polynomial size in the length of the original conjugacy problem.  This now shows that the conjugacy problem in $G_n$ is in NP.

\

One could also compose these operations to find a single closed form for conjugation in general.  Such a closed form, would be not unlike the one computed in the previous section, but altogether much more complicated.  If instead we consider right or left multiplication by syllables, we can obtain closed forms for multiplication of normal forms.  By using these closed forms, we can also perform these algebraic operations with elements represented by exponent vectors in polynomial time.

\section{Reduction of TSSP to SSP}
In this section we show that $SSP \leq_p TSSP$.  To make this easier we introduce another problem $SSP'$ that is similar to $SSP$ and in fact show that $SSP \leq_p SSP' \leq_p TSSP$.  We define $SSP'$ as follows:  given a list of integers $\{k_1, \cdots, k_n\}$ and an integer $M$, decide if there exists a solution to the equation:

$$k_1 x_1 + \cdots k_n x_n = M \,\,\,\, \mbox{where} \,\,\,\, x_1 , \cdots, x_n \in \{-1,0,1\}$$

\begin{lem}
$SSP' \leq_p TSSP$.
\end{lem}
\begin{proof}
Consider $SSP'(\{k_1, \cdots, k_n\}, M)$ and $TSSP(\{0, k_1, 0, k_2, \cdots, 0, k_n\}, M)$, instances of $SSP'$ and $TSSP$ respectively.  In this case, we have that  if $(x_1, x_2, \cdots, x_n)$ is a solution for  $SSP'(\{k_1, \cdots, k_n\}, M)$ then $(y_1, y_2, \cdots, y_{2n})$ is a solution for the corresponding $TSSP$ problem where $y_{2i} = |x_i|$ and $y_{2i - 1} = 1$ if:

\centerline{$x_i = -1$ and $y_1 + \cdots + y_{2i - 2}$ is even ($-k_i$ appears in the sum)}
\centerline{or}
\centerline{$x_i =  1$ and $y_1 + \cdots + y_{2i - 2}$ is odd ($k_i$ appears in the sum)}

and is 0 otherwise.
\end{proof}

It is more work to then show that $SSP \leq_p SSP'$. We adapt a proof from the appendix of \cite{kellerer2004knapsack} by Kellerer, Pferschy, and Pisinger.  Consider the following systems of equations:

\begin{equation} \begin{cases} \sum_{i = 1}^{n} k_i x_i = M \\ x_i \in \{0,1\} \end{cases} \end{equation}
\begin{equation} \begin{cases}  \sum_{i = 1}^{n} k_i x_i = M \\ x_i + y_i = 1  \,\,\,\, \mbox{for} \,\,\,\, i =1,\cdots, n \\ x_i, y_i \in \{-1, 0,1\} \end{cases} \end{equation}
\begin{equation}  \begin{cases} \sum_{i = 1}^{n} (4^{n - i} + 4^n k_i) x_i + \sum_{i = 1}^{n} 4^{n - i} y_i  = 4^n M + 4^n - 1 \\ x_i, y_i \in \{-1, 0,1\} \end{cases} \end{equation}

First, note that $(6)$ and $(7)$ have equivalent solutions: any set of $x_i$ that satisfies one will satisfy the other.  The constraints $x_i + y_i = 1$ and $x_i, y_i \in \{-1, 0,1\}$ prevent $x_i$ from ever being $-1$.  What is less apparent is that $(7)$ and $(8)$ have the same solution set.  If this is the case, we can solve any instance of $SSP$, $(6)$, using an algorithm that solves the equivalent $SSP'$ $(8)$.   If we also show that the size of $(8)$ is only polynomially larger than $(6)$ then we will have in fact shown that $SSP \leq_p SSP'$ and proving that both $SSP$ and $TSSP$ are NP-complete.

\

\begin{prop}
The following systems of equations have the same set of solutions:
\begin{equation}  \begin{cases}  \sum_{i = 1}^{n} k_i x_i = M \\ x_1 + y_1 = 1 \\ x_i, y_i \in \{-1, 0,1\} \end{cases} \end{equation}
\begin{equation}  \begin{cases}  x_1 + y_1 + 4\sum_{i = 1}^{n} k_i x_i = 4M + 1 \\ x_i, y_i \in \{-1, 0,1\} \end{cases} \end{equation}
\end{prop}
\begin{proof}
First note that anything that is a solution to $(9)$ is a solution to $(10)$.  In the other direction, assume that $(x_1, \cdots, x_n, y_1)$ is a solution to $(10)$.  Note that that is equivalent to saying that:

\begin{equation}  4(\sum_{i = 1}^{n} k_i x_i  - M) = 1 - x_1 - y_1 \end{equation}

Using the fact that$-1 \leq 1 - x_1 - y_1\leq 3$ we then get

\begin{equation}  -1 \leq  4(\sum_{i = 1}^{n} k_i x_i  - M) \leq 3   \end{equation}

Finally, since $\sum_{i = 1}^{n} k_i x_i  - M$ is an integer, $(12)$ can only be satisfied if $\sum_{i = 1}^{n} k_i x_i - M = 0$ implying that $(x_1, \cdots, x_n)$ is a solution for $(9)$.
\end{proof}
\begin{prop}
The systems of equations $(7)$ and $(8)$ have the same set of solutions.
\end{prop}

\begin{proof}
As we did in the previous proposition, we combine the conditions $x_i + y_i = 1$ to the equation $\sum_{i = 1}^{n} k_i x_i = M$ to obtain an instance of $SSP'$ whose solution will yield a solution to the corresponding instance of the $SSP$.

\

We then continue as in the proposition, merging our system of equations into just one, by performing the same steps beginning with $x_1 + y_1 = 1$ and ending with $x_n + y_n = 1$.  Note, that as we perform each step, we are not changing the solution set.  After we have performed the first two steps we have the equation:

$$x_2 + y_2 + 4 x_1 + 4 y_1 + 4^2 \sum_{i = 1}^{n} k_i x_i = 4(4M + 1) + 1$$

and then after $n$ steps, we have obtained:

$$\sum_{i = 1}^n 2^{n - 1} x_i + \sum_{i = 1}^n 2^{n - 1} y_i + 4^n \sum_{i = 1}^{n} k_i x_i = 4^n M + 4^{n -1} + 4^{n -2} + \cdots + 1$$

After collecting like terms on the left and summing the geometric series on the right we have $(8)$.

\end{proof}

\begin{thm}
$SSP \leq_p SSP' \leq_p TSSP$ implying that the $TSSP$ is NP-complete and furthermore so is the conjugacy decision problem in the $G_n$.
\end{thm}
\begin{proof}
All that is left to show is that the process described in this section turns instances of $SSP$, $(6)$ into instances of $TSSP$, $(8)$ in polynomial time and that the size of the $TSSP$ instance is only polynomially larger than the $SSP$ instance.  First notice that $(8)$ has $2n$ indeterminates and since each of the $4^i$ can be expressed in $2n$ bits, the number of digits needed to express each coefficient increases linearly.  As such, $(8)$ is polynomially larger than $(6)$.  Additionally, the new coefficients can clearly be computed in polynomial time.
\end{proof}
Furthermore, from the argument in section 3, a polynomial time algorithm for the conjugacy search problem over the $G_n$ would imply P = NP.

\section{Acknowledgments and Support}
The authors would like to thank Vladimir Shpilrain for useful conversations on related topics.  We would also like to thank the anonymous reviewer for their many helpful comments.

\

Delaram Kahrobaei is partially supported by the Office of Naval Research grant N00014120758 and also supported by PSC-CUNY grant from the CUNY research foundation, as well as the City Tech foundation.

\bibliographystyle{plain}
\bibliography{conjnpbib}
\end{document}